\documentclass[12pt]{amsart}
\usepackage{amscd,amssymb,latexsym}
\usepackage{fullpage}
\newcommand{\Mdef}[2]{\newcommand{#1}{\relax \ifmmode #2 \else $#2$\fi}}

\newcommand{\sm }{\wedge}

\newcommand{\tensor}{\otimes}

\newcommand{\sdr}{\rtimes}

\newcommand{\Hom}{\mathrm{Hom}}

\newcommand{\Ext}{\mathrm{Ext}}
\Mdef{\bhom}{\mathbf{\hat{H}om}}

\Mdef{\Mod}{\mathrm{mod}}

\newcommand{\st}{\; | \;}
\newcommand{\hash}{\#}

\newtheorem{thm}{Theorem}[section]
\newtheorem{lemma}[thm]{Lemma}
\newtheorem{prop}[thm]{Proposition}
\newtheorem{cor}[thm]{Corollary}

\theoremstyle{definition}

\newtheorem{defn}[thm]{Definition}

\newtheorem{notation}[thm]{Notation}

\newtheorem{example}[thm]{Example}

\newtheorem{remark}[thm]{Remark}

\newcommand{\qqed}{\qed \\[1ex]}
\renewenvironment{proof}[1][\hspace*{-.8ex}]{\noindent {\bf Proof #1:\;}}{\qqed}

\Mdef{\PH} {\Phi^H}
\Mdef{\PK} {\Phi^K}
\Mdef{\PL} {\Phi^L}
\Mdef{\PT} {\Phi^{\T}}

\Mdef{\ef}{E{\cF}_+}
\Mdef{\etf}{\widetilde{E}{\cF}}
\Mdef{\eg}{E{G}_+}
\Mdef{\etg}{\tilde{E}{G}}

\newcommand{\etp}{\tilde{E}\cP}

\newcommand{\piA}{\pi^{\cA}}

\Mdef{\infl}{\mathrm{inf}}
\Mdef{\defl}{\mathrm{def}}
\Mdef{\res}{\mathrm{res}}
\Mdef{\ind}{\mathrm{ind}}
\Mdef{\coind}{\mathrm{coind}}

\Mdef{\univ}{\mathcal{U}}

\Mdef{\Fp}{\mathbb{F}_p}
\Mdef{\Zpinfty}{\Z /p^{\infty}}
\Mdef{\Zpadic}{\Z_p^{\wedge}}

\newcommand{\dichotomy}[2]{\left\{ \begin{array}{ll}#1\\#2 \end{array}\right.}

\newcommand{\adjunction}[4]{
\diagram
#1:#2 \rrto<0.7ex> &&
#3  \llto<0.7ex> :#4 
\enddiagram}
\newcommand{\lra}{\longrightarrow}
\newcommand{\lla}{\longleftarrow}

\newcommand{\lr}[1]{\langle #1 \rangle}

\newcommand{\Gspectra}{\mbox{$G$-{\bf spectra}}}

\Mdef{\we}{\mathbf{we}}
\Mdef{\fib}{\mathbf{fib}}
\Mdef{\cof}{\mathbf{cof}}
\Mdef{\BI}{\mathcal{BI}}

\newcommand{\colim}{\mathop{  \mathop{\mathrm {lim}} \limits_\rightarrow} \nolimits}

\Mdef{\B}{\mathbb{B}}
\Mdef{\C}{\mathbb{C}}
\Mdef{\D}{\mathbb{D}}
\Mdef{\E}{\mathbb{E}}
\Mdef{\T}{\mathbb{T}}
\Mdef{\F}{\mathbb{F}}
\Mdef{\G}{\mathbb{G}}
\Mdef{\I}{\mathbb{I}}
\Mdef{\N}{\mathbb{N}}
\Mdef{\Q}{\mathbb{Q}}
\Mdef{\R}{\mathbb{R}}
\Mdef{\bbS}{\mathbb{S}}
\Mdef{\Z}{\mathbb{Z}}

\Mdef{\bA}{\mathbb{A}}
\Mdef{\bB}{\mathbb{B}}
\Mdef{\bC}{\mathbb{C}}
\Mdef{\bD}{\mathbb{D}}
\Mdef{\bE}{\mathbb{E}}
\Mdef{\bF}{\mathbb{F}}
\Mdef{\bG}{\mathbb{G}}
\Mdef{\bH}{\mathbb{H}}
\Mdef{\bI}{\mathbb{I}}
\Mdef{\bJ}{\mathbb{J}}
\Mdef{\bK}{\mathbb{K}}
\Mdef{\bL}{\mathbb{L}}
\Mdef{\bM}{\mathbb{M}}
\Mdef{\bN}{\mathbb{N}}
\Mdef{\bO}{\mathbb{O}}
\Mdef{\bP}{\mathbb{P}}
\Mdef{\bQ}{\mathbb{Q}}
\Mdef{\bR}{\mathbb{R}}
\Mdef{\bS}{\mathbb{S}}
\Mdef{\bT}{\mathbb{T}}
\Mdef{\bU}{\mathbb{U}}
\Mdef{\bV}{\mathbb{V}}
\Mdef{\bW}{\mathbb{W}}
\Mdef{\bX}{\mathbb{X}}
\Mdef{\bY}{\mathbb{Y}}
\Mdef{\bZ}{\mathbb{Z}}

\Mdef{\cA}{\mathcal{A}}
\Mdef{\cB}{\mathcal{B}}
\Mdef{\cC}{\mathcal{C}}
\Mdef{\mcD}{\mathcal{D}} 
\Mdef{\cE}{\mathcal{E}}
\Mdef{\cF}{\mathcal{F}}
\Mdef{\cG}{\mathcal{G}}
\Mdef{\mcH}{\mathcal{H}} 
\Mdef{\cI}{\mathcal{I}}
\Mdef{\cJ}{\mathcal{J}}
\Mdef{\cK}{\mathcal{K}}
\Mdef{\mcL}{\mathcal{L}}

\Mdef{\cM}{\mathcal{M}}
\Mdef{\cN}{\mathcal{N}}
\Mdef{\cO}{\mathcal{O}}
\Mdef{\cP}{\mathcal{P}}
\Mdef{\cQ}{\mathcal{Q}}
\Mdef{\mcR}{\mathcal{R}}
\Mdef{\cS}{\mathcal{S}}
\Mdef{\cT}{\mathcal{T}}
\Mdef{\cU}{\mathcal{U}}
\Mdef{\cV}{\mathcal{V}}
\Mdef{\cW}{\mathcal{W}}
\Mdef{\cX}{\mathcal{X}}
\Mdef{\cY}{\mathcal{Y}}
\Mdef{\cZ}{\mathcal{Z}}

\Mdef{\ca}{\mathcal{a}}
\Mdef{\ct}{\mathcal{t}}

\Mdef{\At}{\tilde{A}}
\Mdef{\Bt}{\tilde{B}}
\Mdef{\Ct}{\tilde{C}}
\Mdef{\Et}{\tilde{E}}
\Mdef{\Ht}{\tilde{H}}
\Mdef{\Kt}{\tilde{K}}
\Mdef{\Lt}{\tilde{L}}
\Mdef{\Mt}{\tilde{M}}
\Mdef{\Nt}{\tilde{N}}
\Mdef{\Pt}{\tilde{P}}

\Mdef{\tA}{\tilde{A}}
\Mdef{\tB}{\tilde{B}}
\Mdef{\tC}{\tilde{C}}
\Mdef{\tE}{\tilde{E}}
\Mdef{\tH}{\tilde{H}}
\Mdef{\tK}{\tilde{K}}
\Mdef{\tL}{\tilde{L}}
\Mdef{\tM}{\tilde{M}}
\Mdef{\tN}{\tilde{N}}
\Mdef{\tP}{\tilde{P}}

\Mdef{\ft}{\tilde{f}}
\Mdef{\xt}{\tilde{x}}
\Mdef{\yt}{\tilde{y}}

\Mdef{\Ab}{\overline{A}}
\Mdef{\Bb}{\overline{B}}
\Mdef{\Cb}{\overline{C}}
\Mdef{\Db}{\overline{D}}
\Mdef{\Eb}{\overline{E}}
\Mdef{\Fb}{\overline{F}}
\Mdef{\Gb}{\overline{G}}
\Mdef{\Hb}{\overline{H}}
\Mdef{\Ib}{\overline{I}}
\Mdef{\Jb}{\overline{J}}
\Mdef{\Kb}{\overline{K}}
\Mdef{\Lb}{\overline{L}}
\Mdef{\Mb}{\overline{M}}
\Mdef{\Nb}{\overline{N}}
\Mdef{\Ob}{\overline{O}}
\Mdef{\Pb}{\overline{P}}
\Mdef{\Qb}{\overline{Q}}
\Mdef{\Rb}{\overline{R}}
\Mdef{\Sb}{\overline{S}}
\Mdef{\Tb}{\overline{T}}
\Mdef{\Ub}{\overline{U}}
\Mdef{\Vb}{\overline{V}}
\Mdef{\Wb}{\overline{W}}
\Mdef{\Xb}{\overline{X}}
\Mdef{\Yb}{\overline{Y}}
\Mdef{\Zb}{\overline{Z}}

\Mdef{\db}{\overline{d}}
\Mdef{\hb}{\overline{h}}
\Mdef{\qb}{\overline{q}}
\Mdef{\rb}{\overline{r}}
\Mdef{\tb}{\overline{t}}
\Mdef{\ub}{\overline{u}}
\Mdef{\vb}{\overline{v}}

\Mdef{\hc}{\hat{c}}
\Mdef{\he}{\hat{e}}
\Mdef{\hf}{\hat{f}}
\Mdef{\hA}{\hat{A}}
\Mdef{\hH}{\hat{H}}
\Mdef{\hJ}{\hat{J}}
\Mdef{\hM}{\hat{M}}
\Mdef{\hP}{\hat{P}}
\Mdef{\hQ}{\hat{Q}}

\Mdef{\thetab}{\overline{\theta}}
\Mdef{\phib}{\overline{\phi}}

\Mdef{\uA}{\underline{A}}
\Mdef{\uB}{\underline{B}}
\Mdef{\uC}{\underline{C}}
\Mdef{\uD}{\underline{D}}

\Mdef{\bolda}{\mathbf{a}}
\Mdef{\boldb}{\mathbf{b}}
\Mdef{\bfD}{\mathbf{D}}

\Mdef{\fm}{\frak{m}}
\Mdef{\fp}{\frak{p}}

\newcommand{\fX}{\mathfrak{X}}

\Mdef{\eps}{\epsilon}

\newcommand{\up}{\mathsf{V}}

\newcommand{\cell}{\mathrm{Cell}}

\input{xypic}
\usepackage{graphicx}
\newcommand{\sub}{\mathrm{Sub}}
\newcommand{\PP}{\mathbb{P}}
\newcommand{\cEi}{\cE^{-1}}
\newcommand{\cIi}{\cI^{-1}}
\newcommand{\Hbar}{\overline{H}}
\newcommand{\bbT}{\mathbb{T}}
\newcommand{\Zt}{\tilde{\Z}}
\newcommand{\Qt}{\tilde{\Q}}
\newcommand{\gb}{\overline{g}}
\newcommand{\cotoral}{\leq_{ct}}
\newcommand{\fa}{\mathfrak{a}}
\newcommand{\Wbar}{\overline{W}}
\newcommand{\End}{\mathrm{End}}
\renewcommand{\tb}{\overline{\times}}
\newcommand{\nt}{\mathbf{nt}}
\newcommand{\full}{\mathrm{full}}

\newcommand{\cNbar}{\overline{\cN}}
\newcommand{\Ah}{\hat{A}}
\newcommand{\wt}{\tilde{w}}
\newcommand{\Th}{\mathrm{Th}}
\newcommand{\diag}{\mathrm{diag}}
\newcommand{\spl}{\mathrm{Split}}
\newcommand{\gammaperp}{\gamma^{\perp}}
\newcommand{\width}{\mathrm{width}}
\newcommand{\cOcF}{\mathcal{O}_{\cF}}
\newcommand{\cOcFt}{\widetilde{\mathcal{O}}_{\cF}}
\newcommand{\Vt}{\tilde{V}}
\newcommand{\e}{(e)}
\newcommand{\Qmod}{\Q\mbox{-mod}}
\newcommand{\cVGfull}{\cV^G_{\full}}
\newcommand{\Rt}{\tilde{R}}
    \newcommand{\cospan}{\lrcorner}
    \newcommand{\module}{\mbox{-mod}}
\newcommand{\cFa}{\cF_a}
\newcommand{\cFb}{\cF_b}
\newcommand{\cOcFa}{\cO_{\cFa}}
\newcommand{\cOcFb}{\cO_{\cFb}}
\newcommand{\cKa}{\mathcal{K}_a}

\newcommand{\cOcKa}{\cO_{\cKa}}

\newcommand{\cOcK}{\cO_{\cK}}

\newcommand{\siftyV}[1]{S^{\infty V( #1)}}
\newcommand{\supV}[1]{S^{\up ( #1)}}
 \newcommand{\efp}{E\cF_+}
      \newcommand{\efap}{E\cF^a_+}
      \newcommand{\efbp}{E\cF^b_+}
      \newcommand{\cEhi}{\hat{\cE}^{-1}}
      \newcommand{\Sh}{\hat{S}}

\newcommand{\freeGspectra}{\mbox{free-$G$-spectra}}
\newcommand{\cofreeGspectra}{\mbox{cofree-$G$-spectra}}
\newcommand{\HBGmodules}{\mbox{$H^*(BG)$-mod}}
\newcommand{\torsionHBGmodules}{\mbox{tors-$H^*(BG)$-mod}}
\newcommand{\completeHBGmodules}{\mbox{comp-$H^*(BG)$-mod}}
\newcommand{\moduleGspectra}{\mbox{-mod-$G$-spectra}}

\newcommand{\modulespectra}{\mbox{-mod-spectra}}
\newcommand{\modules}{\mbox{-mod}}
\newcommand{\spectra}{\mathrm{spectra}}

\newcommand{\cSi}{\cS^{-1}}
\newcommand{\ebar}{\overline{e}}
\newcommand{\Tt}{\tilde{T}}
\newcommand{\cOcFZ}{\cO_{\cF/Z}}
\newcommand{\cOcFTt}{\cO_{\cF/\Tt}}
\newcommand{\piGs}{\pi^G_*}

\newcommand{\mcRt}{\tilde{\mcR}}

\newcommand{\piG}{\pi^G}

\begin{document}
\title{Rational $G$-spectra for rank 2 toral groups of mixed type}

\author{J.P.C.Greenlees}
\address{Mathematics Institute, Zeeman Building, Coventry CV4, 7AL, UK}
\email{john.greenlees@warwick.ac.uk}

\date{}
\begin{abstract}
  In this paper we give an explicit and calculable
  algebraic model for the block of
  rational $G$-spectra on full subgroups 
  when $G$ has identity component a 2-torus $\T$, and
component group of order 2 acting non-trivially on $H_1(\T)$. The
example of particular interest is the normalizer of the maximal torus
in $U(2)$, which constitutes one of the most complicated blocks in the analysis
of $SU(3)$. This builds on the determination of subgroups 
up to conjugacy in \cite{t2wqalg}. 
\end{abstract}

\thanks{The author is grateful for comments, discussions  and related
  collaborations with S.Balchin, D.Barnes, T.Barthel, M.Kedziorek,
  L.Pol, J.Williamson. The work is partially supported by EPSRC Grant
  EP/W036320/1. The author  would also  like to thank the Isaac Newton
  Institute for Mathematical Sciences, Cambridge, for support and
  hospitality during the programme Equivariant Homotopy Theory in
  Context, where later parts of  work on this paper was undertaken. This work was supported by EPSRC grant EP/Z000580/1.  } 

\maketitle

\tableofcontents

\section{The shape of abelian models}

\subsection{Context}
For  a compact Lie group $G$, it has been conjectured \cite{AGconj} that the category of 
rational $G$-spectra is Quillen equivalent to differential graded 
objects of an abelian category $\cA (G)$. In this paper we  make 
$\cA (G)$ explicit for some 2-dimensional groups $G$, which necessarily
have identity component a torus.

Toral groups
$$1\lra \T \lra G\stackrel{\pi}\lra W\lra 1 $$
are classified by the $\Z W$-module $\Lambda_0=H_1(\T)$ and an extension 
class. The rational representation $\Lambda_0^{\Q}=H_1(\T;\Q)$ describes the
general shape. We are concerned with the special case that $\T$ is a
2-torus and $W$ is of order 2, so there are three possibilities:
$\Q^2, \Q\oplus\Qt $ and $\Qt^2$. The first case is a   central 
extensions of a torus, and the last case has no fixed circles; these
are more straightforward and covered elsewhere \cite{AGnoeth,
  gqwf}.

The present paper deals with the case
$\Lambda_0^{\Q}=\Q\oplus \Qt$. In this case we have either
$\Lambda_0=\Z \oplus \Zt$ or $\Lambda_0=\Z W$. Since $H^2(W; \Z\oplus
\Zt)$ is of order 2, there are two groups of the first type (namely
$G\cong O(2)\times T$ or $G\cong Pin(2)\times T$), and since $H^2(W;\Z W)=0$
the second case is automatically split, with $G\cong T^2\sdr
C_2$ and $C_2$ exchanging the factors.
The present paper proves the conjecture in  these three mixed cases of rank 2. 

\subsection{The shape of the model}
In general the abelian category $\cA (G)$ is a space of sheaves over
the space $\fX_G=\sub(G)/G$ of conjugacy classes of (closed) subgroups
of $G$. The space occurs as the Balmer spectrum of finite rational
$G$-spectra \cite{spcgq}, so it has two relevant topologies:
the h-topology (the quotient of the Hausdorff metric topology on $\sub
(G)$) and the Zariski topology. The
closed sets in the Zariski topology are the sets that are h-closed and
closed under passage to cotoral subgroups \footnote{$K$ is cotoral in
  $H$ if $K$ is normal in $H$ with quotient a torus. A conjugacy class $(K)$
  is cotoral in a conjugacy class $(H)$ if the relation holds for some
  representative subgroups.}

We note that every subgroup $H$ of $G$ is either contained in the
identity component $\T$, or else it is full in the sense that $\pi
H=W$. It is easy to see this gives a partition of $\fX_G$ into Zariski
clopen subsets: 
$$\fX_G =\cV^G_1\amalg \cV^G_W$$
where 
$$\cV^G_1=\sub (\T)/W \mbox{ and } \cV^G=\{ (H)\st H \mbox{ is full
}\}.  $$
This means 
$$\cA (G)=\cA (G|\fX_G)=\cA(G|\cV^G_1)\times
\cA(G|\cV^G_W). $$
Since the toral part $\cV^G_1$ is  covered in \cite{AGtoral, gtoralq}, we
focus on the full component $\cV^G_W$. 

 The spaces $\cV^G_W$ for the three groups $G$  are different in detail but 
their general structures are the same. In all three cases the dimension comes half from limits of subgroups  
with finite Weyl groups and half from a cotoral inclusion. 
In all three  cases $\fX_G$ is  approximately (but not exactly) the product of two one 
dimensional pictures.

 \subsection{Associated work in preparation}
This paper is the third in a series  of 5 constructing an algebraic
category $\cA (SU(3))$ and showing it gives an algebraic model for 
rational $SU(3)$-spectra. This series gives a concrete 
illustrations of general results in small and accessible 
examples.

The first paper \cite{t2wqalg} describes the group theoretic data that feeds into the construction of an
abelian category $\cA (G)$ for all toral groups $G$ and makes them
explicit for toral subgroups of rank 2 connected groups. In
particular, this determines the space $\fX_G$ in all three of the
cases considered here.

The second paper (which is not used here) constructs algebraic models for all relevant 1-dimensional
blocks. 

The paper \cite{u2q} assembles the  information from the first three
papers and that
from \cite{gtoralq} to give an abelian category $\cA (U(2))$ in 7
blocks and shows it is an algebraic model for rational
$U(2)$-spectra. Finally, the paper \cite{su3q} constructs $\cA (SU(3))$
in 18 blocks and shows it is equivalent to the category of rational
$SU(3)$-spectra. The most complicated parts of the model for $G=U(2)$
and $G=SU(3)$ are the toral blocks, which are based on the work in the
present paper.

This series is part of a more general programme. Future installments
will consider blocks with Noetherian Balmer spectra \cite{AGnoeth} and
those with no cotoral inclusions \cite{gqwf}. 
An account of the general nature of the models is in preparation
\cite{AVmodel}, and the author hopes that this will be the basis of the proof that the
category of rational $G$-spectra has an algebraic model in general.

\subsection{Contents}
We begin in Section \ref{sec:abmixed} by describing the algebraic
model $\cA (G|\full)$. The brief Section \ref{sec:strategy}
summarises the strategy for showing that $DG-\cA(G|\full)$ is
equivalent to rational $G$-spectra over full subgroups. Section  
\ref{sec:isotropicrings} constructs the appropriate structured and
isotropically formal  ring spectra used in the proof, and Section
\ref{sec:decompS} shows that they correspond to algebraic adelic
constructions and  the sphere can be reconstructed from
them. Section  \ref{sec:genform} discusses generators for the category. Section 
\ref{sec:piA} gives an equivalence with an algebraic module category
and Section \ref{sec:CST} finds an economical cellular skeleton of
the module category in which homology isomorphisms are the only
weak equivalences and establishes it gives a model.

\section{Abelian models for mixed components}
\label{sec:abmixed}

We identified the space of subgroups and all relevant structure in
Sections 9 and 11 
of \cite{t2wqalg}, and we will use the same notation here. 

In this section we describe the standard model. We will describe the
separated and complete models elsewhere. 
 In Sections \ref{sec:isotropicrings} to 
\ref{sec:piA}  we show that  the model gives an algebraic model
of  $G$-spectra over full subgroups. The final piece of algebra,
showing that the standard model is a cellular skeleton of the
appropriate module category, is deferred to Section \ref{sec:CST}, so that we can motivate it using by cells from topology.


\subsection{The groups}
We are dealing with three isomorphism classes of groups: two with toral
lattice $\Lambda_0=\Z\oplus \Zt$ (one split and one non-split) and one with
toral lattice $\Lambda_0=\Z W$. The spaces $\cV=\cV^G_{\full}$ of full subgroups share a common form thanks to the fact
that the rational toral lattice is $\Q \oplus \Qt$ in all cases.

The Lie algebra $L\T$ is the direct sum of the $+1$  and $-1$
eigenspaces, and  we will write $Z$ (for centre) and $\Tt$ for their
image circles under
the exponential map. For the toral lattice $\Z\oplus\Zt$ we have
$\T=Z\times \Tt$; indeed, $G\cong Z\times O(2)$ in the split case and
$G\cong Z\times Pin(2)$ in the non-split case. For the toral lattice $\Z W$ the two circles
intersect in their subgroup of order 2; indeed we may take $G$ to be the
normalizer of diagonal matrices in $U(2)$ and the tori intersect in
$\{\pm I\}$.

\subsection{The subgroups}
First we recall the general form of $\cV^G_\full$ from Sections 9 and 11 of \cite{t2wqalg}.
The general picture is that for each $1\leq m, n\leq \infty$ there are
finitely many subgroups $H^{\lambda}(m,n)$ (one, two or three
subgroups depending on $G, m$ and $n$). In fact if $m=\infty$ there is exactly
one subgroup $H(\infty, n)$ for each $n$ and if  $n=\infty$  there are
precisely two groups $H^\lambda (m, \infty)$ for each $m<\infty$, 
indexed by $\lambda \in \{s, ns\}$ (for `split' and `non-split'). If $m$ and $n$ are both finite,
then there are either 2 or 3 subgroups $H^\lambda (m,n)$. In all cases there
are two subgroups we call Type 1, one of them split and one
nonsplit;  these are indicated by $\lambda=(1,s)$ or $(1,ns)$. The pairs
$(m,n)$ for which there is a third subgroup (Type 2) with $\lambda=2$
are distributed in a regular pattern. For $\Z\oplus \Zt$, there are three subgroups
when $m$ and $n$ are both even. For $\Z W$ there are three subgroups
when $m+n $ and $m-n$ are both even. 

Each of the subgroups $H^\lambda (m,n)$ is cotoral in $H(\infty, n)$ and is a
subgroup of exactly one of the subgroups $H^\lambda (m, \infty)$
($H^{s}_1(m, n)$ and $H_2(m, n)$ lie in $H^s(m, \infty)$,
and $H^{ns}_1(m, n)$ lies in $H^{ns}(m, \infty)$), The subgroups
$H^\lambda (m, \infty)$ are cotoral in $H(\infty, \infty)=G$.  We draw the
picture for the particular case $\Lambda_0=\Z W$. 

\begin{equation*}
\resizebox{\displaywidth}{!}
{\xymatrix{
H(\infty, 1)&H(\infty,2) &H(\infty,3)&\cdots  & H(\infty, \infty)\\
\vdots &\vdots&\vdots&&\vdots\\
H_1^{s/ns}(4,1)&H_1^{s/ns}(4,2), H_2(4,2)&H_1^{s/ns}(4,3)&\cdots 
&H^{s/ns}(4,\infty)\\
H_1^{s/ns}(3,1), H_2(3,1)&H_1^{s/ns}(3,2)&H_1^{s/ns}(3,3), H_2(3,3)&\cdots 
&H^{s/ns}(3,\infty)\\
H_1^{s/ns}(2,1)&H_1^{s/ns}(2,2), H_2(2,2)&H_1^{s/ns}(2,3)&\cdots 
&H^{s/ns}(2,\infty)\\
H_1^{s/ns}(1,1), H_2(1,1)&H_1^{s/ns}(1,2)&H_1^{s/ns}(1,3), H_2(1,3)&\cdots 
&H^{s/ns}(1,\infty) 
}}
\end{equation*}

\subsection{The Thomason  filtration}
We are considering the space of full subgroups of $G$. 
The pattern is visible from the pictures we have drawn. 
The group $G$ is the unique point of height 2,
there are three lines of height 1, two vertical cotoral lines on the
right and horizontal one at the top that his purely topological.
The height 0 points can be arranged in lines below
the points with finite Weyl groups, and form a 2 dimensional panel.

\begin{lemma}
    The   Thomason height of a subgroup $H$ is equal to the dimension of $H$.    
  The subgroups with finite 
  Weyl groups are precisely those containing $Z$. 

  The space subgroups of Thomason height 1 is partitioned into
  $$\Th_1^+:=\up Z\setminus \up G=\{ H\st Z\subseteq H \neq G\} \mbox{
    and } \Th_1^-=\up \Tt\setminus \up G:=\{ H\st \Tt\subseteq H \neq G\}  $$
\end{lemma}

\begin{remark}
By \cite[11.5]{prismatic},  the fact that 
Thomason height is equal to the dimension holds for any toral group
for which  all simple summands of the rational toral lattice are one
dimensional.
\end{remark}

\begin{proof}
  This may all be observed directly from the calculations in Sections
  9 and 11 of \cite{t2wqalg} as summarized above.

The statement about finite Weyl groups is clear since every subgroup
$H$ is cotoral in $HZ$.

Finite subgroups are obviously cotorally minimal and they are  h-isolated (even) amongst all 
subgroups of $T^2$, so are of Thomason height 0.

The $W$-invariant subgroups containing a circle contain 
$Z$ or $\Tt$, and are therefore not isolated,  so the groups of
Thomason height 0 are precisely the finite subgroups. 

The one dimensional subgroups are obviously cotorally minimal amongst
the remainder, and h-isolated by the classification. Since $G$ is not
h-isolated, we see that the one dimensional groups are of Thomason
height 1 and partitioned as stated. 
  \end{proof}

\subsection{Neighbourhoods}
For each subgroup $H$ we need  a system of
neighbourhoods of $H$ amongst subgroups  with finite Weyl group.
In fact  we choose a system  of neighbourhoods $\cN_H^\alpha$  of $H$
in $\sub(H)/H$ for positive integers $\alpha$,  so that their images
$\cNbar_H^\alpha$ in $\sub(G)/G$ give neighbourhoods of $H$. We then take their
cotoral downclosures $\cU_H^\alpha:=\Lambda \cNbar_H^\alpha$.

\begin{itemize}
\item Subgroups of Thomason height 0 are isolated, so the neighbourhoods are
  singletons and the horizontal map gives no data.
 \item Subgroups of Thomason height 1, with finite Weyl
group (those of the form $H(\infty , n)$) are isolated amongst subgroups with finite
Weyl group, so we take $\cNbar=\{ H\}$ and
$$\cU_H=\{ K \st K \mbox{
  cotoral in } H\}=\{ K \st KZ=H\} . $$
This consists of the groups in the column below it in the pictures,
which is the subgroups  of the form $H^\lambda(m,n)$ for this fixed
$n$, where $m, \lambda$ take all available values. 
\item Subgroups of Thomason height 1 with infinite Weyl group (those
  of the form $H^{s/ns}(m, \infty)$) take
  $\cU_H^\alpha=\cNbar_H^\alpha$
  to consist of subgroups of order $\geq \alpha$ from amongst
  $H^{\lambda}_1(m,n)$ with the same superscript $\lambda$  together with
  $H_2(m,n)$ in  the split case where relevant.
\item The only subgroup of Thomason height 2 is $G$, where we take
  $\cNbar_G^\alpha$ to consist of subgroups $H(\infty ,n)$ with $n\geq
  \alpha$, so that $\cU_G$ consists of groups $H^\lambda (m,n)$ with $n\geq
  \alpha$ where $m, \lambda$ take all  available values. 
  \end{itemize}



\subsection{The inflation diagram of rings} 
The standard model $\cA (G|\full)$ consists of certain `sheaves' with
additional structure. In fact, we will just
specify the values on four particular sets $U$. These four sets
correspond to subsets of the set of simple representations of $W$
occurring in $H_1(\T; \Q)$: since there two simple representations
$\Q,\Qt$, there are four subsets and four sets $U$. In other language,
they correspond to the four connected $W$-invariant subgroups of $G_e$
$$\up G, \up Z, \up \Tt, \up 1. $$
The rationale for this is that these are naturally occurring sets, each
containing groups that can be conveniently described together. We note
it is consistent with what was done for tori, where the standard model
described the values on sets parametrised by the connected subgroups
$K$ of $G$, where the set corresponding to $K$ is  
$\up K=\{ H \st H\supseteq K\} $. We will show in due course 
 that $\cA$ is a Cellular
Skeleton for the model structure.

For the  standard model we have the following data. First, we have the inflation
diagram of rings
$$\xymatrix{
&\cO_{\cF/G}\drto \dlto&\\
\cOcFZ \drto &&\cOcFTt\dlto\\
&\cOcF&
}$$
where 
$$\cO_{\cF/G}=\Q, \cOcFZ=\prod_{H(\infty, n)}\Q ,
\cOcFTt=\prod_{H^{s/ns}(m, \infty)} \Q [c], 
\cOcF=\prod_{H^{\lambda}(m,n)}\Q[c].$$

The inflation maps from $\cO_{\cF/G}$ are given by the unit.

The inflation map 
$\cOcFZ\lra \cOcF$ is given by components. 
The component indexed by  the finite  subgroup $F=H^{\lambda}(m,n)$ is 
given by choosing the unique subgroup containing $Z$ and $F$ and 
then using the corresponding projection $\cOcFTt\lra H^*(BW_G^e(F\cdot 
Z)$. (The group $F\cdot Z$ in question is $H ( \infty, n)$). 

 The inflation map 
$\cOcFTt\lra \cOcF$ is given by its components. 
The component indexed by  the finite  subgroup $F=H^{\lambda}(m,n)$ is 
given by choosing the unique minimal subgroup containing $\Tt$ and $F$ and 
then using the corresponding projection $\cOcFTt\lra H^*(BW_G^e(F\cdot 
\Tt))$. (The group $F\cdot \Tt$ in question is 
$H^{\lambda'} (m, \infty)$ where $\lambda'=s$ if $\lambda\in \{ (1,
s),  2\}$
and $\lambda'=ns$ if $\lambda=(1, ns)$).

\subsection{The component structure}
The corresponding component structure is indicated by the diagram  
$$
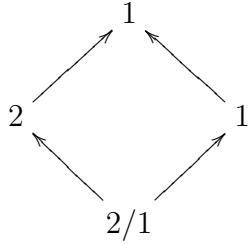
\xymatrix{
&1&\\
2\urto &&1\ulto \\
&2/1\urto \ulto&}$$
The two alternatives at the bottom occur only for $\Lambda_0=\Z\oplus \Zt$. In all
cases the homomorphisms are non-trivial wherever possible.

\subsection{Euler classes}
For the Type 1 rings $\cOcFTt$ and $\cOcF$ we write $\cE$ for the
mutliplicatively closed set of Euler classes. More precisely, for
$\cOcF$ this is the set of $c^v$ where $v: \cF\lra \N$ is a function
zero almost everywhere. Similarly for $\cOcFTt$ it is the set of $c^v$
where  $v: \cF\Tt\lra \N$ is a function
zero almost everywhere. 

For the Type 0 subring $\cOcFZ$, we write $\cI$ for the
multiplicatively closed set of characteristic functions of cofinite sets
sets, which are the functions $i: \cF Z\lra \{0,1\}$ which are 1 almost
everywhere. 

\subsection{The standard model}
An object of the standard model is a diagram of $\cOcFt$-modules, where
$$
\cOcFt=\left(
\begin{gathered}
\xymatrix{
&\cEi\cIi \cOcF&\\
\cEi \cOcF\urto &&\cIi \cOcF\ulto \\
&\cOcF\urto \ulto&}
\end{gathered}\right)$$
Thus $N$ is the diagram
$$\xymatrix{
&N(\up G)&\\
N(\up Z)\urto &&N(\up \Tt)\ulto \\
&N(\up 1)\urto \ulto&}$$
which is quasicoherent and extended
\begin{enumerate}
\item {\em quasicoherence} is the requirement that 
$$N(\up Z)=\cEi 
  N(\up 1), N(\up \Tt)=\cIi 
  N(\up 1), N(\up G)=\cEi\cIi   N(\up 1)$$
\item {\em extendedness} is the requirement that 
$$
N(\up Z)=\cEi \cOcF\tensor_{\cOcFZ} \phi^ZN, 
N(\up \Tt)=\cIi \cOcF\tensor_{\cOcFTt} \phi^{\Tt} N,
N(\up G)=\cIi\cEi \cOcF\tensor \phi^GN, $$
where $\phi^ZN$ is an $\cOcFZ$-module,
$\phi^{\Tt} N$ is an $\cOcFTt$-module,
$\phi^GN$ is an $\Q$-module, 
\end{enumerate}

\section{The abelian models are Quillen models: general strategy}
\label{sec:strategy}

 The structure of the argument is the same as that for tori in 
\cite{tnqcore}: we show that the  sphere spectrum is the pullback of 
rings which are isotropically concentrated and formal in a strong 
sense.

The core of the proof is the fact that  the sphere $S^0$ is a homotopy pullback 
of a punctured cube of isotropically simpler ring spectra. 

This allows us to outline the proof: we give a symbolic description 
and then discuss the steps.       
      \begin{multline*}
        \Gspectra|\cV 
\stackrel{0}\simeq S^0\module- \Gspectra|\cV 
      \stackrel{1}\simeq   \cell ((S^0)^{\cospan}\module -\Gspectra)\\
    \stackrel{2}  \simeq  \cell (((\widetilde{S^0})^{\cospan})^{G}\module-\cW-\spectra) 
    \stackrel{3}  \simeq  \cell (C^{G}_*((\widetilde{S^0})^{\cospan})\module 
    -\Q [\cW] \module)\\
    \stackrel{4}  \simeq  \cell (\piG_*((\widetilde{S^0})^{\cospan})\module -\Q [\cW]\module) 
     \stackrel{5} \simeq DG-\cA (G|\cV) 
   \end{multline*}

Equivalence 0 simply uses the fact that $G$-spectra are 
modules over the sphere spectrum. Equivalence 1 uses the fact 
\cite[4.1]{diagrams} that the category of modules over a homotopy pullback ring is equivalent to the 
cellularization of the category of generalized diagrams over the 
individual modules, together with the fact that the localized sphere 
is the pullback.

From Equivalence 2 onwards,  we introduce variants on the terms of the initial cube 
so as to keep track of  the finite Weyl groups. The cospan 
$(\widetilde{S^0})^\cospan$ of $G$-spectra replaces terms of 
$(S^0)^\cospan$ by coinductions which vary by subgroup. 
The new objects have  the property that their $G$-fixed point spectra are 
products of spectra with homology $H^*(BW_G^e(K))$ for relevant 
subgroups $K$ and the category of modules take values in the corresponding product of 
categories with $W_G^d(K)$-action. Equivalence 2 is based on the fact \cite{modfps} that taking fixed points under a connected group gives an equivalence.

Equivalence 3 follows from Shipley's Theorem \cite{ShipleyHZ}, and is 
easily adapted to the type of diagram we have. Equivalence 4 is a formality 
statment,  again requiring proof in each case. Finally, Equivalence 5 
folows from the Cellular Skeleton Theorem, which will 
identify the cellularization of the algebraic category of modules with 
the derived category of an abelian category. Again this needs 
discussion in each case. 

\section{Isotropically simple ring spectra}
\label{sec:isotropicrings}
The key to the proof is to express the sphere spectrum as  a pullback
of a cube of commutative ring spectra which are istropically simple
and formal. In this section we introduce these ring spectra and
establish their basic properties.

\subsection{Geometric localizations}
For a space $X$ and a representation $V$, we may form the  suspension
$S^V\sm X$, and the inclusion $S^0\lra S^V$ induces multiplication by
the Euler class of $V$ in a complex orientable context. Passing to
limits, we may form $\siftyV{H}=\bigcup_{V^H=0}S^V$, which corresponds
to inverting the Euler classes of $H$-essential representations
$\cE_H=\{e(V)\st V^H=0\}$. Analagously,   if $U$ is an open and closed
subset of the space $\Phi G$ of conjugacy classes of subgroups with
finite Weyl group, there is an idempotent in the Burnside ring $A(G)$ with support $U$
and we may form $e_UX$. If  $V$ is an intersection of open and closed
subsets we may form 
$$\ebar_VN=\colim_{U\supseteq V}e_UN.  $$
In particular, this applies to any subgroup $K$ with finite Weyl
group.

\begin{lemma}
\label{lem:VZisotropy}
The support of $\ebar_GS^0$ is the set $\up (\Tt)$ of subgroups containing $\Tt$. 
\end{lemma}

\begin{proof}
This is immediate from the supports of the idempotents. 
\end{proof}

\begin{lemma}
\label{lem:VZisotropy}
The support of $\siftyV{Z}$ is the set $\up (Z)$ of subgroups containing $Z$. 
\end{lemma}

\begin{proof}
By definition $(\siftyV{Z})^{Z}=0$ and so the support of
$\siftyV{Z}$ contains $V(Z)$. On the other hand, if $H$ is a full
1-dimensional subgroup of $G$ not containing $Z$, then it is normal and $G/H$ is either a
circle or $T\times C_2$, we may choose a $Z$-essential
representation of 
$G/H$. If $H$ is a finite group then it is a
subgroup of a 1-dimensional subgroup not containing $Z$.
\end{proof}

\begin{lemma}
\label{lem:VGisotropy}
The support of $\siftyV{G}=\siftyV{Z}\sm \ebar_GS^0$ is the singleton
$\up (G)=\{ G\}$.\qqed
\end{lemma}

\begin{remark}
We occasionally use the notation $S^{\cC}$ for the couniversal 
space of the cofamily $\cC$, so that we have 
$$\siftyV{Z}=\supV{Z}, \siftyV{G}=\supV{G}, \ebar_GS^0=\supV{\Tt}.$$
The visual pun relating $V(K)$ to $\up (K)$ is intended to be suggestive.  
\end{remark}

\begin{cor}
(i) The object $e_nS(\infty V(Z))_+$ has support consisting of 
subgroups  $H^\lambda (m ,n)$ for $1\leq m <\infty$. For each such subgroup $H^\lambda(m,n)$
there is an idempotent selfmap of $e_nS(\infty V(Z))_+$ and the summand is 
$E\lr{H^\lambda(m,n)}$. 

(ii) The object $\ebar_G S(\infty V(Z))_+$ has support consisting of 
subgroups  $H^\lambda (m ,\infty )$ for $1\leq m <\infty$. For each such subgroup 
$H^\lambda(m,\infty )$
there is an idempotent selfmap of $\ebar_G S(\infty Z)$ and the summand is 
$E\lr{H^\lambda(m,\infty)}$. 
\end{cor}

This enables us to construct objects with the required supports.

\subsection{Commutative ring spectra}
Now we need to construct some commutative ring spectra. It will be sufficient 
for them to be commutative in the naive sense (i.e., the least commutative of the 
$N_\infty$ options). 

The first construction is that any couniversal space has a commutative 
model. Thus if $\cC$ is a cofamily, the space $E\cC$ (which could be 
constructed as $S^0*E\cC^c$) is a commutative ring. 
 
We note that spaces of the form $X_+$ have a diagonal, so that $F(X_+, 
R)$ is a commutative ring spectrum for any commutative ring spectrum 
$R$. 
This in particular applies when $X$ is a  universal space $E\mcH$ for a family 
of subgroups $\mcH$. Given an idempotent self map of $E\mcH_+$ we may 
localize $DE\mcH_+$ at $eDE\mcH_+$  to obtain a (naive-)commutative 
ring spectrum. 

We need to describe the diagram
$$\xymatrix{
&\mcR (\up G)\drto \dlto&\\
\mcR(\up Z)\drto &&\mcR(\up \Tt)\dlto\\
&\mcR(\up 1)&
}$$

We first give the answer, and then give a detailed discussion of 
the notation and choices in turn. 
$$\xymatrix{
&\siftyV{Z}\sm \ebar_GS^0\drto \dlto&\\
\siftyV{Z}\sm DE\lr{\cF/Z} \drto &&\ebar_GDE\lr{\cF/\Tt}\dlto\\
&DE\lr{\cF}&
}$$

\begin{lemma}
For the set $\up G$, we may take the commutative ring 
$$\mcR(\up G)=\siftyV{Z}\sm \ebar_GS^0\simeq \etp, $$
and 
$$\piG_*(\mcR (\up G))=\Q. $$
\end{lemma}

\begin{proof}
The isotropy lemmas in the previous section establish the equivalence, 
from which the homotopy is immediate. 
\end{proof}

For the set $\up Z$, we note $Z$ is normal, and we would probably expect the expression 
$\siftyV{Z} \sm DE\cF/Z_+$, where $DE\cF/Z_+$ is constructed in the
group $G/Z=O(2)$ and then inflated to $G$. This is correct, except
that the only finite subgroups of $O(2)$ that are relevant are the
dihedral ones. We may therefore consider $DE\cF/Z_+$ itself and then
localize using the idempotent $(1-e_{toral})$ supported on the
non-toral subgroups. Alternatively we may avoid this discussion and
directly take  $DE\lr{\mcD}\simeq \prod_nDE\lr{D_{2n}}$, noting that
we may simplify this somewhat since $DE\lr{D_{2n}}\simeq
e_{D_{2n}}S^0$.  In any case
the essential fact is that the spectrum is constructed as an
$O(2)$-spectrum and inflated. 

\begin{lemma}
For the set $\up Z$, we may  take the commutative ring 
$$\mcR(\up Z)=\siftyV{Z}\sm DE\lr{\mcD} $$
and 
$$\piG_*(\mcR (\up Z))=\prod_{D\in \mcD}\Q. $$
\end{lemma}

\begin{proof}
We have used Lemma \ref{lem:VZisotropy}
$$[X, \siftyV{Z}\sm Y]^G=[\Phi^Z X, \Phi^Z Y]^{G/Z}, $$ 
from which the homotopy is immediate. 
\end{proof}

For $\up \Tt$ we proceed differently, since $\Tt$ is not normal. In
fact we take the ring spectrum $DE[Z\not \subseteq ]_+=DS(\infty
V(Z))_+$ with coisotropy consisting of subgroups not containing $Z$
and then localize to cut out those not containing $\Tt$. Finally, this
is interpreted in the full localization, so that non-full subgroups do
not contribute. 

\begin{lemma}
For the set $\up \Tt$, we may take the commutative ring 
$$\mcR(\up \Tt)=\ebar_G DS(\infty V(Z))_+$$
and 
$$\piG_*(\mcR (\up \Tt))=\prod_{H\supseteq \Tt, H\not = G}\Q [c]. $$
\end{lemma}

\begin{proof}
  We calculate
$$  \piG_*(\mcR (\up \Tt)=[S(\infty Z)_+, S^0]^G_*=[(S(\infty
Z)/Z)_+,S^0]^{G/Z}=
H^*_{O(2)}((S(\infty Z)/Z)_+), $$
 from which the homotopy is immediate. 
\end{proof}

For $\up 1$ it is clear we wish to use all finite isotropy, but of
course just involving full finite subgroups. All finite subgroups have
a  Weyl group whose identity component is a circle. 

\begin{lemma}
For the set $\up 1$, we may  take the commutative ring 
$$\mcR(\up 1)=DE\cF_+$$
and 
$$\piG_*(\mcR (\up 1))=\prod_{F finite }\Q [c]. $$
\end{lemma}

\begin{proof} For each factor, we may calculate
  $$\piG_*(DE\lr{F})=[E\lr{F}, S^0]^G_*=[E\lr{F}, S^0]^{N_G(F)}_*
  =[E\lr{F}, S^0]^{N_G(F)/F}_*
  =[BW_G(F), S^0]$$
from which the homotopy is immediate. 
\end{proof}

\section{Decomposing the sphere}
\label{sec:decompS}
In the previous section we introduced a number of isotropically simple ring
spectra. We now show that they can be used to reconstruct the sphere
in an adelic fashion, and that the homotopical localizations
correspond to algebraic localizations in homotopy groups.
\subsection{The pullback cube}
The appropriate cube here is a massively simplified version of the one
for the 2-torus. The simplification is that there are only two
subgroups of codimension 1 instead of infinitely many.

\begin{prop}
The cube 

  \begin{equation*}
  \resizebox{\displaywidth}{!}
  {\xymatrix{
&\mcR(\up Z)\times \mcR(\up \Tt)\ddto \rrto&&\siftyV{G}\sm\mcR(\up Z)\times
\siftyV{G}\sm \mcR(\up \Tt)\ddto \\
S^0\rrto \ddto \urto  &&\mcR (\up G)\ddto \urto&\\
&\siftyV{Z}\sm \mcR(\up 1)\times \ebar_G\mcR(\up 1)
\rrto&&\siftyV{G}\sm\mcR(\up 1)\times
\siftyV{G}\sm \mcR(\up 1)\\
\mcR (\up 1)\urto \rrto &&\siftyV{G}\sm \mcR(\up 1)\urto&
}}
\end{equation*}
is a homotopy pullback in the category of $G$-spectra over $\cV$.
\end{prop}

\begin{proof}
It suffices to take the fibre of the map from the left to right and
show the resulting square is a pullback. In view of the fact that the
fibre of $S^0\lra \siftyV{G}$ is $E\cP_+$, it suffices to show the
square from the right hand end
$$\xymatrix{
S^0\rto \dto& \mcR (\up Z)\times \mcR(\up \Tt)\dto\\
\mcR (\up 1)\rto &\siftyV{Z} \sm \mcR (\up 1)\times \ebar_G \mcR(\up 1) 
}$$
is an $H$ equivariant pullback square for each proper subgroup $H$ of
$G$. If $H$ is finite, the two right hand rings are contractible, and
the left hand one is the $\cF$-equivalence $S^0\lra D\efp$. If $H$ is
one dimensional it contains either $Z$ or $\Tt$, but not both. If $H$
contains $Z$ we have
$$\xymatrix{
S^0\rto \dto& \mcR (\up Z)\dto\\
\mcR (\up 1)\rto &\siftyV{Z} \sm \mcR (\up 1)
}$$
which is the well known pullback square for the circle group. 

If $H$ contains $\Tt$, we have 
$$\xymatrix{
  S^0\rto \dto& \mcR(\up \Tt)\dto\\
\mcR (\up 1)\rto &\siftyV{Z} \sm \mcR (\up 1)\times \ebar_G \mcR(\up 1) 
}$$
which is the well known pullback square for the dihedral part of $O(2)$. 
\end{proof}

\subsection{Localized diagrams}
The connection to the algebraic model is given by the coefficients
together with the effect on modules of various
localizations. 

\begin{lemma}
\label{lem:qc}
(i) For modules $N$ over $\mcR (\up 1)=D\efp$, we have 
$$\piG_*(\siftyV{Z}\sm N)=\cEi \piG_*(N) \mbox{ and }
\piG_*(\ebar_G N)=\cIi \piG_*(N)$$

(ii) For modules $N$ over $\mcR (\up Z)=\siftyV{Z}\sm DE\cF/Z_+$, we have 
$$\piG_*(\ebar_G M)=\cIi \piG_*(M) $$

(iii) For modules $P$ over $\mcR (\up \Tt )=\ebar_GDS(\infty V(Z))_+$, we have 
$$\piG_*(\siftyV{Z}\sm P)=\cEi \piG_*(P). $$
\end{lemma}

\begin{proof}
  The arguments for all three parts are similar. We focus on Part (i).

  Quasicoherence for $N(\up 1)\lra N(\up Z)$ comes 
from the fact that each inclusion $S^0\lra S^V$ induces 
mulitiplication by $e(V)$ provided $V$ is complex orientable.

Quasicoherence for $N(\up 1)\lra N(\up \Tt)$ comes 
from the fact that $\supV{\Tt}$ is defined as a colimit of 
multiplication by idempotents.
\end{proof}

\begin{cor}
  
  \begin{multline*}
    \piG_*
\left( \begin{gathered}
\xymatrix{
&\ebar_G \siftyV{Z} \sm D\efp&\\
\siftyV{Z}\sm D\efp \urto&&\ebar_GD\efp\ulto\\
&D\efp\urto \ulto&
}
\end{gathered}
\right) 
=\\
\left( \begin{gathered}
\xymatrix{
&\cEi\cIi \cOcF&\\
\cEi \cOcF \urto&&\cIi \cOcF\ulto\\
&\cOcF\urto \ulto&
}
\end{gathered}
\right)
\end{multline*}

$$\piG_*
\left( \begin{gathered}
\xymatrix{
&\ebar_G \mcR(\up Z)&\\
\mcR(\up Z)\urto&&\ast \ulto\\
&\ast \urto \ulto&
}
\end{gathered}
\right) 
=
\left( \begin{gathered}
\xymatrix{
&\cIi \cOcFZ&\\
\cOcFZ \urto&&0\ulto\\
&0\urto \ulto&
}
\end{gathered}
\right)$$

$$\piG_*
\left( \begin{gathered}
\xymatrix{
& \siftyV{Z} \sm \mcR(\up \Tt)&\\
\ast \urto&&\mcR (\up \Tt)\ulto\\
&\ast\urto \ulto&
}
\end{gathered}
\right) 
=
\left( \begin{gathered}
\xymatrix{
&\cEi \cOcFTt&\\
0\urto&&\cOcFTt\ulto\\
&0\urto \ulto&
}
\end{gathered}
\right)$$

\end{cor}

\subsection{The coinduced diagram}
Finally, we need to identify the coinduced diagram (for extended
discussion of the coinduction process, see \cite{gq1}). 
We need to describe the diagram
$$\xymatrix{
&\mcRt (\up G)\drto \dlto&\\
\mcRt(\up Z)\drto &&\mcRt(\up \Tt)\dlto\\
&\mcRt(\up 1)&
}$$

This is designed to account for groups not being normal and for the
Weyl group not being connected. To deal with the component structure
where the discrete quotient $W_G^d(K)$ of $W_G(K)$ acts, we let $N_G^f(K)$ denote the
inverse image of $W_G^e(K)=N_G(K)/K$ in $N_G(K)$, so that
$N_G(K)/N_G^f(K)=W_G^d(K)$.  Thus, at a subgroup $K$, the ring $\mcR(K)$ gets
replaced by the ring $\mcRt (K)=F_{N^f_G(K)}(G_+,\mcR(K) )$ so that
$$\mcRt(K)^G=\mcR(K)^{N^f_G(K)}. $$

For full normal
subroups $N_G^f(G)=G$, and hence $\mcRt (K)=\mcR(K)$, and this applies
to all subgroups in $\up \Tt$, so that 
$$\mcRt(\up G)=\mcR (\up G)=\ebar_G \siftyV{Z} \mbox{ and }
\mcRt(\up \Tt)=\mcR (\up \Tt)=\ebar_G D(S(\infty V(Z))_+).$$

On the other hand, for any finite subgroup $F=H^{\lambda}(m,n)$ has
normalizer $N_G(H^{\lambda}(m,n))=H(\infty , 2n)$ with
$N_G^f(H^{\lambda}(m,n))=H(\infty , n)$ so that 
$$\mcR (\up 1)=\prod DE\lr{F} \mbox{ and }\mcRt (\up 1)=\prod 
F_{N^f_G(F)}(G_+, DE\lr{F}). $$

Similarly, any subgroup containing $Z$ is of the form $H(\infty, n)$
with $N_G(K)=N_G^f(K)=H(\infty ,2n)$ and 
$$\mcR (\up Z)=\siftyV{Z}\sm DE\cF/Z_+=\siftyV{Z} \sm \prod
DE\lr{K} \mbox{ and }  \mcRt (\up Z)=\siftyV{Z} \sm \prod 
F_{K}(G_+, DE\lr{K}). $$

\begin{lemma}
For $K=1, Z, \Tt, G$ we have  $\piG_*(\mcRt(\up K))=\piG_*(\mcR (\up K))$. 
\end{lemma}

\begin{proof}
  For definiteness, we take $K=1$, and calculate one of the terms. The
  other arguments are similar. Writing $N=N_G(F), N^f=N_G(F),
  W^d=W_G^d(F)$, we have 
  $$\piG_*(DE\lr{F}) =[E\lr{F}, S^0]^G\stackrel{\cong}\lra
  [E\lr{F}, S^0]^N =([E\lr{F}, S^0]^{N^f})^{W^d}, $$
  but the action of $W_G^d$ is trivial in each case. 
    \end{proof}

\section{Generators and formality}
\label{sec:genform}
We have established the link between topology and algebra at the
level of ring spectra and their homotopy groups. The resulting
diagrams are shown to be formal in Subsection \ref{subsec:formal}
below. 

We also need to identify
the realisable modules in the two contexts and show they correspond,
which we do by identifying generators. First we consider cells
in Subsection \ref{subsec:gens} and then dual cells in Subsection \ref{subsec:dualcells}.

\subsection{Generators}
\label{subsec:gens}
With the appropriate ring spectra in place, generators are  easily verified.
In each case we have an adjunction relating $R$-module $G$-spectra
with modules over a product of polynomial rings each with an action of
a finite group. It remains only to identify the image of a set of
generators. It is convenient to write $\sigma_K=G_+\sm_K e_KS^0$ when 
$K$ is isolated in $\Phi K$.

To prepare expectations we remind the reader of some easier examples

\begin{example}
(i) If $G$ is connected and $R=DEG_+$
then $R$ generates the category of $R$-module $G$-spectra.

(ii) If $G=O(2)$
and $K=D_{2n}$ and $R=DE\lr{K}$ then the category of $R$-module
$G$-spectra is generated by $R\sm G/K_+$ or by $R\sm
\sigma_K$.

(iii) Suppose $G=O(2)$ and $R=eS^0$ where $e $ is the dihedral
idempotent.  The category of $R$-module $G$-spectra is generated by $eG/D_{2n})_+$
for all $n$ together with $eS^0$ or equally by $\sigma_{2n}$ together
with $\ebar S^0$.
\end{example}

Prepared with these examples, the case of direct concern is
straightforward. 

\begin{prop}
  \label{prop:Rmodgen}
  
(i) For $R=\mcR(\up G) =\etp$, the category of $R$-module $G$ spectra is generated by
$R$.

(ii)  For $R=\mcR (\up H(\infty ,n))=\siftyV{Z}\sm DE\lr{n}$, the 
category of $R$-module $G$ spectra is generated by $R$ and $R\sm 
G/H(\infty, n)_+$.  For $R=\mcR (\up Z)=\siftyV{Z}\sm D\efp$,   the 
category of $R$-module $G$ spectra is generated 
by $R$ together with $R\sm \sigma^G_{H(\infty,n)}=R\sm \sigma^{O(2)}_{D_{2n}}$ for all $n$.

(iii)  For $R=\mcR (\up \Tt)$  the category 
of $R$-module $G$ spectra is generated by $R$.

(iv)  For $R=\mcR (\up F)=DE\lr{F}$  the category of $R$-module $G$ spectra is 
generated by $R\sm \sigma_F $. For $R=\mcR (\up 1)=\prod_F DE\lr{F}$
the category of $R$-module $G$ spectra is
generated by $R\sm \sigma_F $ together with $R\sm
\overline{\sigma}^{s/ns}_{m,\infty}$ or by $R\sm
\sigma^{s/ns}_{(1,n)}, R\sm \overline{\sigma}_{(1,\infty)}$.
\end{prop}
\begin{proof}
The category of $R$-modules is always generated by the objects $R\sm
G/H_+$ as $H$ runs through conjugacy classes of subgroups. It remains
only to observe that many of these are redundant. 

We use two facts. Firstly, if $H$ has no subgroups in the geometric
isotropy of $R$ then $R\sm G/H_+\simeq *$. This leaves only $G/G_+$
for Part (i), only $G/H_+ $ for subgroups containing $Z$ in Part
(ii), and $G/H_+$ for Subgroups containing $\Tt$ in Part (iii).
Secondly, if $L$ is cotoral in $K$ then $G/K_+\sm R $ builds $G/L_+\sm
R$. We prove this by induction on the rank of $K/L$. If $K/L$ is a
circle we have a cofibre sequence $K/L_+=S(\alpha)_+\lra S^0\lra
S^{\alpha}$. All the spectra $R$ have Thom isomorphisms, so $R\sm
S^{\alpha}\simeq \Sigma^2 R$, and hence $R\sm K/K_+$ builds $R\sm
K/L_+$. Inducing to $G$ gives the required result. 
  \end{proof}

  \subsection{Algebraic dual cells}
  \label{subsec:dualcells}
 In the algebraic context we will need to identify
 cellularization. Any set of generators will do, so we use the dual
 cells since it is simpler to write down their homotopy since no
 shifts are involved. We record
 this calculation here.

\begin{defn}
\label{defn:DsigmaH}
The algebraic dual cells are as follows. 

(i) The $G$-fixed dual cell is 
  $$D\sigma_G=DS^0=S^0=\left(
\begin{gathered}
\xymatrix{
  &\cEi \cIi \cOcF\tensor \Q&\\
  \cEi \cOcF\tensor_{\cOcFZ}\cOcFZ  \urto &&\cIi \cOcF\tensor_{\cOcFTt}\cOcFTt\ulto\\\
  &\cOcF \urto \ulto&
  }
\end{gathered}
\right) 
$$

(ii) The dual basic cell corresponding to a 1-dimensional subgroup $H$
containing $Z$  is 
  $$D\sigma_H=\left(
\begin{gathered}
\xymatrix{
  &0&\\
  \cEi \cOcF\tensor_{\cOcFZ}\Q_H  \urto &&0 \ulto\\\
  &\prod_{(F)\subseteq H}  \Q_F [c]\urto \ulto&
  }
\end{gathered}
\right) 
$$

(iii) The dual basic cell corresponding to a 1-dimensional subgroup $H$
containing $\Tt$  is 
  $$D\sigma_H=\left(
\begin{gathered}
\xymatrix{
  &0&\\
0\urto &&  \cIi \cOcF\tensor_{\cOcFTt}\Q_H  \ulto \\
  &\prod_{(F)\subseteq H} \Q_F \urto \ulto&
  }
\end{gathered}
\right) 
$$

(iv) The dual basic cell corresponding to a 0-dimensional subgroup $F$
is 
  $$D\sigma_F=\left(
\begin{gathered}
\xymatrix{
  &0&\\
0\urto &&  0  \ulto \\
  &\Q_F \urto \ulto&
  }
\end{gathered}
\right)$$
\end{defn}

\subsection{Formality}
\label{subsec:formal}
With the appropriate ring spectra and coefficient behaviour
established, this is easily verified.

\begin{lemma}
\label{lem:formalityone}
The diagram of $G$-spectra is intrinsically formal. 
\end{lemma}

\begin{proof}
The argument is the same as for the torus. We work down from the
height 2 groups. At each stage there are products of localizations of
the previous stage. Products are dealt with by \cite[10.1(ii)]{tnqcore} and
localizations by \cite[10.1 (iii)]{tnqcore}. The shape of the diagram is such
that we can argue inductively as in \cite[10.2]{tnqcore}. The actions of the finite groups are
not obstacle by Maschke's Theorem. 
\end{proof}

\section{From $G$-spectra to the abelian models in mixed components}
\label{sec:piA}

The proof strategy is described in Section \ref{sec:strategy}. These
cases are 2-dimensional, not Noetherian but simpler than that for the 2-torus,
since there are only finitely many relevant connected subgroups.
In fact we deal with the following examples 
\begin{itemize}
\item $G=O(2)\times T$, $\Lambda_0=\Z \oplus \Zt$
\item $G=Pin(2)\times T$, $\Lambda_0=\Z \oplus \Zt$
 \item $G=N_{U(2)}(T^2)$, $\Lambda_0=\Z W$
\end{itemize}

Since our principal concern is the third case, where the 
toral lattice is $\Lambda_0=\Z W$, we will illustrate the argument 
with this example, but we only  use methods that apply 
in apply in all three cases.

We will see the category  $\cA (G|\full)$ is of injective dimension 2, and hence 
$DG-\cA(G|\full)$ admits the injective model structure with 
homology isomorphisms as weak equivalences. 

\begin{thm}
  For $G=O(2)\times T, Pin(2)\times T$ or $N_{U(2)}(T^2)$ above there are Quillen equivalences 
    $$\Gspectra|\full \simeq DG-\cA (G|\full)$$
    \end{thm}

  \subsection{The homology functors}

We describe the functor $\piA_*: \Gspectra \lra \cA (G|\full)$ to the
standard model.

\begin{defn} 
  In effect the functor $\piA_*$ is defined by applying
  $\piG_*(D\efp \sm X\sm (\cdot))$  to the diagram 
$$\xymatrix{
  &\siftyV{G}&\\
  \siftyV{Z}\urto &&\supV{\Tt}\ulto\\
  &S^0\urto \ulto&
}$$

  More precisely 
\begin{multline*}
    \left(
\begin{gathered}
\xymatrix{  &  N(\up 1)&\\
  N(\up Z)\urto&&  N(\up \Tt)\ulto\\
  &N(\up G)\ulto\urto&
}
\end{gathered}
\right)=\\
\left(
  \begin{gathered}
 \xymatrix{
  &  \piGs (X\sm D\efp)&\\
\piGs (X\sm D\efp\sm \siftyV{Z})   \urto&&
 \piGs (\ebar_G X\sm D\efp)  \ulto\\
  &\piGs (X\sm D\efp\sm \siftyV{G})  \ulto\urto&
}
\end{gathered}
\right)
\end{multline*}

This is quasi-coherent by Lemma \ref{lem:qc}, and we will establish
extendedness in Lemma \ref{lem:e} below with
$$
\xymatrix{
  &  V=\pi_*(\Phi^GX)\ar@{--}[dr] \ar@{--}[dl]&\\
  \phi^Z\piA_*(X)=\pi_*(\Phi^ZX)\ar@{--}[dr]&&\phi^{\Tt}(X)=\piG_*(\ebar_GX)\ar@{--}[dl]\\
&\phi^1N=\piGs (X\sm D\efp\sm \siftyV{G}) &
}
  $$
\end{defn}

\begin{lemma}
\label{lem:e}
The object $\piA_*(X)$ is qce and hence lies in the standard model. 
\end{lemma}

\begin{proof}
There are three extendedness isomorphisms. In each case,
we have natural comparison maps from the external to the
internal tensor products, and we argue individually. At $\up G$ the
comparison map reads
$$\cEi \cIi \piGs(D\efp )\tensor \pi_*(\Phi^GX)\lra 
\piGs(\siftyV{G}\sm D\efp \sm \Phi^GX);$$
we comment that $\siftyV{G}\sm \Phi^GX$ can be made sense of either as 
by use of inflation to view the $G/G$-spectrum $\Phi^GX$ as a $G$-spectrum, or else by direct construction of a $G$-spectrum. 
is clearly an isomorphism when $\Phi^GX$ is a sphere, and the general
case follows since both sides are homology theories. 

At $\up Z$, the comparison map is
$$\cEi \piGs(D\efp )\tensor_{\cOcFZ} \pi^{G/Z}_*(\Phi^ZX)\lra 
\piGs(\siftyV{Z}\sm D\efp \sm \Phi^ZX);  $$
we comment that $\siftyV{Z}\sm \Phi^ZX$ can be made sense of either as 
by use of inflation to view the $G/Z$-spectrum $\Phi^ZX$ as a
$G$-spectrum, or else by 
direct construction of a $G$-spectrum. 
The map is obviously an isomorphism for $S^0$, and it follows that it
is an equivalence for $e_US^0$ for each idempotent. These are
generators for the category of spectra over $Z$, so the result
follows. 

Finally at $\up \Tt$, we have the map
$$ \cIi \piGs(D\efp )\tensor_{\cOcFTt} \piGs(\ebar_G X)\lra 
\piGs(\supV{\Tt} \sm D\efp \sm \ebar_G X);$$
we comment that $\ebar_G$ is idempotent, so
$\supV{\Tt}\sm \ebar_G X=\ebar_GS^0\sm \ebar_GX\simeq \ebar_GX. $
This is obviously an isomorphism for $X=S^V$ for any complex
representation of $G$, and these generate the objects over $\Tt$.
\end{proof}

\section{The Cellular Skeleton Theorem in mixed components}
\label{sec:CST}

We have seen at the homotopy level that the category or rational
$G$-spectra with full isotropy is equivalent 
to the cellularization of the category of punctured cubical diagrams 
of modules (where we recall that cellularization may be taken with
respect to the dual basic cells of Definition \ref{defn:DsigmaH}). 
It remains to show that the corresponding fact holds at the level of
abelian categories, so that the standard model 
$\cA(G|\full)$ is a cellular skeleton. Logically this work is
purely about the algebraic categories, but
it is much easier to motivate and explain with the link to spectra
made explicit. 

\begin{prop}
\label{prop:CST}
The cellularization of the category of 
$\piG_*((S^0)^\cospan)[\cW]$-modules has $\cA (G|\full)$ as a cellular skeleton. 
\end{prop}
    
The proof of this will be eased by exploring the category a little. 

\subsection{The adelic cube and the square of simples}

First of all, we need to establish that the standard model corresponds
to the adelic model. The standard model is based on the set of
simple representations of
$W$ occurring in $H_1(\T; \Q)$; there are two of these, so the diagram is a
square. On the other hand, the adelic cube is based on the Thomason
height; there are three heights ($0,1, 2$), and so the adelic cube is
3 dimensional.  We need to explain how the standard model is viewed as
a subcategory of the category of cubical diagrams. This is reasonably 
clear once we display the cubical diagram. It is helpful to first 
display the punctured cube of spectra. 

  \begin{equation*}
  \resizebox{\displaywidth}{!}
  {\xymatrix{
 &\siftyV{Z}\sm DE\cF/Z_+ \sm X \vee \ebar_G \sm X \rrto \ddto 
 &&\ebar_G \siftyV{Z}\sm DE\cF/Z_+ \sm X \vee \ebar_G\siftyV{Z} \sm X  \ddto\\
  &&\siftyV{G} \sm X \ddto \urto&\\
  &\siftyV{Z}\sm DE\cF_+ \sm X \vee \ebar_G \sm D\efp\sm  X \rrto &&
  \ebar_G \siftyV{Z} \sm D\efp \sm X   \\
  D\efp \sm X \rrto \urto&&\ebar_G \siftyV{Z}\sm D\efp \sm  X\urto&
}}
\end{equation*}

Taking homotopy groups we have 
\begin{equation*}
  \resizebox{\displaywidth}{!}
{\xymatrix{
 &\phi^Z N\oplus \phi^{\Tt}N \rrto \ddto &&\cIi \cOcFZ \tensor 
 \phi^ZN\oplus 
 \cEi\cOcFTt\tensor \phi^{\Tt} N\ddto\\
  && N(\up G) \ddto \urto&\\
  &N(\up Z)\oplus N(\up \Tt) \rrto &&N(\up G)\\
N(\up 1) \rrto \urto&& N(\up G)\urto&
}}
\end{equation*}

From this diagram we see how the square of simples gets incorporated
into the adelic cube of heights. 

\subsection{Right adjoints to evaluation}
Since objects of the models are given by diagrams of modules, it is natural to consider the
evaluations at various subgroups. The key to analysing the category is
given by importing objects from these module categories. In effect
these are done through right adjoints to evaluation, but their
properties are a little stronger than this would imply. 

For subgroups $H$ of $G$ and suitable modules $M$ we will construct
objects  $f_H(M)$ of the standard model, which are supported on $H$
and its subgroups. We will first give the construction and then explain their properties. 

\begin{defn}
  \label{defn:rightadjoints}
 At $G$,  for any vector space $W$ we define
$$f_G(W)=
\left(
\begin{gathered}
    \xymatrix{
    & \cEi \cIi \cOcF \tensor W &\\
    \cEi \cIi \cOcF \tensor W \urto&& \cEi \cIi \cOcF \tensor W\ulto\\
    &\cEi \cIi \cOcF \tensor W \urto \ulto&
    }
\end{gathered}
\right) $$

At a 1-dimensional subgroup $H$ containing $Z$, for each vector space 
$W$ with an action of the group of order 2,  we define
$$f_Z(W)=
\left(
\begin{gathered}
    \xymatrix{
    & 0&\\
    \cEi \cOcF \tensor W_H \urto&& 0\ulto\\
    &\cEi \cOcF \tensor W_H \urto \ulto&
    }
\end{gathered}
\right) $$

At a 1-dimensional subgroup $H$ containing $\Tt$, for each
torsion $\Q [c]$-module $M$, we define
$$f_{\Tt}(M_H)=
\left(
\begin{gathered}
    \xymatrix{
    & 0 &\\
    0 \urto&& \cIi \cOcF \tensor_{\cOcFTt} M_H\ulto\\
    &\cIi \cOcF \tensor_{\cOcFTt} M_H    \urto \ulto&
    }
\end{gathered}
\right) $$

At a finite subgroup $F$, for each torsion $\Q[c]$-module $M_H$ with an
action of $W_G^d(F)$ we define
$$f_1(M_H)=
\left(
\begin{gathered}
    \xymatrix{
    & 0&\\
    0\urto&&0\ulto\\
    &M_F \urto \ulto&
    }
\end{gathered}
\right) $$
\end{defn}

From the definitions it is easy to verify the required properties. 
\begin{lemma}
(i)  For any vector space $W$, the functor  $f_G$ has the property 
  $$\Hom (N,f_G(W))=\Hom (N(\up G), W)$$

  (ii)  If $H$ is a 1-dimensional subgroup containing $Z$, 
   then for any vector space $W$ with an action of the component group 
   $\cW_H$ there is a natural isomorphism 
$$\Hom (N,f_G(W))=\Hom (N(\up Z), W)^{\cW_H} .$$

(iii)  If $H$ is a 1-dimensional subgroup containing $\Tt$, 
then   for any torsion $\Q[c]$-module $M$,
there is a natural isomorphism 
  $$\Hom (N,f_H(M))=\Hom (N(\up \Tt ), M)$$

(iv)  If $F$ is a finite subgroup then, for any torsion $\Q[c]$-module 
$M$ with an action of the component group $\cW_H$
  there is a natural isomorphism 
$$\Hom (N,f_H(M))=\Hom (N(\up \Tt ), M)^{\cW_H}$$
\end{lemma}

\begin{remark}
It follows that the functors $f_H$ take injective modules to injective
objects of the standard model. We will use this below to show the
injective dimension of the standard model is precisely 2. 
\end{remark}



\subsection{Proof of Cellular Skeleton Theorem}
Having made the necessary preparations, we now prove Proposition
\ref{prop:CST}.

By inspection we see that each of the dual basic  cells  in Definition
\ref{defn:DsigmaH} lies in the standard
model, essentially because the homotopical localizations are obtained by inverting a 
multiplicatively closed set. More precisely, the value at $\up 1$ for
the dual cell $DG/H_+$ is $\piGs (F_H(G_+, D\efp ))\cong \pi^H_*(D\efp)$.

Conversely, it remains to show that every object of the standard model is 
cellular. We prove this by showing each object of $\cA (G|\full)$ has a finite resolution by 
injectives, each of which is cellular.

\begin{lemma}
The modules $f_G(W)$,  $f_Z(W_H)$, $f_{\Tt}(M_H )$, $f_1(M_F)$  described above are qce and hence in the standard model. 
\end{lemma}

\begin{proof}
  We examine the four objects in Definition \ref{defn:rightadjoints},
  and observe they are all qce. Each of them has a permitted locally
  constant   action  by $\cW$, and this is preserved by direct sums. 
\end{proof}

Finally, we show that the standard model is big enough.
The argument is a 2-dimensional version of a standard 
argument. As a warmup we recall the proof in the simplest case, 
that of  semifree $\T$-spectra. 

\begin{example}
We suppose given the diagram $\Q \lra \Q[c,c^{-1}]\lla \Q [c]$ of rings 
we see that any object $X=(V\lra P \lla N)$ is cellularly equivalent to an 
object of the standard model as follows. 

First, there is a map $X\lra 
e(V)$ to a cellular object, so it suffices to show the fibre $X'$ is 
cellularly equivalent to an object in the standard model. In other 
words we are reduced to the case with $V=0$. 

Now consider the fibre sequence of modules, $\Gamma N\lra N\lra 
N[1/c]$. The object $f_1(\Gamma N)$ is cellular, so it suffices to see 
$f_1(N[1/c])$ is cellularly trivial. 

The two dual cells are (i) $D\sigma_1=(0\lra 0\lla \Q)$ and (ii) 
$D\sigma_G=(\Q \lra \Q[c,c^{-1}]\lla \Q[c])$. The torsion cell 
$D\sigma_1$  clearly has no maps to the injective object 
$f_1(N[1/c])$. The maps from $D\sigma_G$ is the homotopy pullback 
$$\xymatrix{
[D\sigma_G,  X]\rto \dto &V\dto \\
N\rto &P 
}$$ 
In our case $V=0$ and $N\simeq P$, so the homotopy pullback is 
trivial. 
\end{example}

\begin{lemma} 
Any module is cellularly equivalent to an object
of the standard model. 
\end{lemma}

\begin{proof}
We start with an arbitrary object $X$. This is really a diagram on the
punctured cube of non-empty subsets of $\{0,1,2\}$, but for legibility we
record only $0,1, 2, 01, 12, 012$ in the pattern
$$\xymatrix{
2\drto &&\\
&21\drto&\\
2\urto \drto&&210\\
&10\urto&\\
0\urto &&
}$$

For rings we have
$$\xymatrix{
\Q\drto &&\\
&\cIi \cOcFZ \times \cEi\cOcFTt\drto&\\
\cOcFZ \times \cOcFTt 
\urto \drto&&\cEi \cIi\cOcF\\
&\cEi\cOcFTt\urto&\\
\cOcF \urto &&
}$$
For the object $X$ we have 
$$\xymatrix{
W\drto &&\\
&A\drto&\\
N 
\urto \drto&&Z \\
&B\urto&\\
Q\urto &&
}$$

Proceeding with the proof, we note that $f_G$ is right
adjoint to evaluation on
the full module category so there is a map $X\lra f_G(W)$ with cofibre
$X'$ zero at $G$. Since $f_G(W)$ is cellular it suffices to show $X'$
is cellularly equivalent to an object of the standard model. In other
words, we may assume $W=0$.  

Now we construct a fibre sequence 
$$\Gamma_ZX'\lra X'\lra \cEi X'$$ 
and note $\Gamma_ZX'=f_Z(T_Z)$ where $T_Z$ is $\cE$-torsion and hence 
$f_Z(T_Z)$ is cellular, and we take $X''=\cEi X'$. 
Now we construct a fibre sequence 
$$\Gamma_{\Tt}X''\lra X''\lra \cIi X''$$ 
and note $\Gamma_{\Tt}X''=f_{\Tt}(T_{\Tt})$ where $T_{\Tt}$ is $\cI$-torsion and hence 
$f_{\Tt}(T_{\Tt})$ is cellular, and we take $X'''=\cEi X''$.

We thus have an object $X'''$ which is zero at $G$ but where $\cE$ and $\cI$
are invertible everywhere it makes sense.

It remains to prove that $X'''$ is cellularly equivalent to 0. It is
clear that  $[\sigma_F, X''']^G=0$ for $F$ finite since $\phi^1X'''$
is torsion free and injective. Similarly $[\sigma_H, X''']^G=0$
for subgroups with identity component $Z$ because all terms are
$\cI$-torsion free and injective, and  $[\sigma_H, X''']^G=0$
for subgroups with identity component $\Tt$ because all terms are
$\cE$-torsion free and injective.

Finally for $G$ homotopy of $X$ we have a homotopy pullback cube where the
maps come from the original diagram for $X$ and the cube is 
$$\xymatrix{
&N\rrto \ddto &&A\ddto \\
[\sigma_G,X]^G_*\urto \rrto\ddto &&V\urto \ddto\\ 
&B\rrto  &&Z \\
Q\urto \rrto&&Z\urto^=  
}$$
For $X'''$ we see the horizontal maps on the back face are
equivalences, as is the map $Q\lra B$. Now consider the square of
fibres from left to right. Since $V=0$ it shows that $[\sigma_G,
X]^G_*$ is a pullback of three zero groups and hence zero. 
\end{proof}

\subsection{The injective dimension of the standard model}
It is quite instructive to construct injective resolutions in general,
which has the further merits of preparing
 the way for computation and establishing the existence
of model structures. 

\begin{lemma}
The standard model $\cA (G|\full)$ is of injective dimension $\leq 2$. 
\end{lemma}

We will give an explicit resolution whose terms are realizable by
spectra, so we will obtain the required corollary.

\begin{cor}
All objects of $\cA $ are cellular. 
\end{cor}

\begin{proof}
  For an arbitrary object $X$ we construct an exact sequence
$$0\lra X \lra I_0X\lra I_1X\lra I_2X\lra 0$$
with $I_sX$ injective. The aim is to ensure that (a) $I_1X$ is zero at
$G$ and at $Z$, (b) $I_1X$ is injective at
$\Tt$, (c) $I_2X$ is zero at $G,Z$ and $\Tt$ and (d) $I_2$ is injective at $1$.

To start with we will take
$$I_0X=f_G(J_{G,0})\oplus f_Z(J_{Z,0})\oplus
f_{\Tt}(J_{\Tt,0})\oplus f_1(J_{1,0}) , $$
where the modules $J_{K,0}$ are to be described, and
are injective over their respective rings. 

First we take $J_{G,0}=X(\up G)$,
which is automatically injective, and the map $X\lra f_G(J_{G,0})$ corresponds to
the identity. Next, at $Z$ we take $J_{Z,0}=\ker (X(\up Z)\lra X(\up
G))$ (again automatically 
injective) with the component mapping to $f_Z(J_{Z,0})$ corresponding 
to the identity.   At $\Tt$, we take $J_{\Tt,0}$ to be the injective 
envelope of the torsion module $\ker (X(\up Z)\lra X(\up G))$
with the component mapping to $f_Z(J_{\Tt,0})$ corresponding 
to the map into the injective hull. Finally, at $1$ we take $J_{1,0}$ to be the injective 
envelope of the torsion module $\ker (X(\up 1)\lra X(\up Z)\times
X(\up \Tt))$ with the component mapping to $f_1(J_{1,0})$ corresponding 
to the map into the injective hull. 

The cokernel $C_1$ of $X\lra I_0X$ is zero at $G$ and $Z$, and is divisible
at $\Tt$ and $1$. We may therefore 
 take
$$I_1X=f_{\Tt}(J_{\Tt,1})\oplus f_1(J_{1,0}) , $$
where the $J_{K,1}$ are injective over their respective rings. 
As noted the cokernel $C_1$ is injective at $\Tt$ so we may take
$J_{\Tt,1}=C_1(\Tt)$. At $1$, we take $J_{1,1}$ to be the injective
module $C_1(1)$. 

Finally, the cokernel $C_2$  of the map $I_0X\lra I_1X$ is concentrated at 1,
where it is a quotient of the divisible module $C_1(\Tt)\oplus C_1(1)$
which is therefore divisible and hence injective. We therefore have $C_2=f_1(C_2(1))$.
\end{proof}

\begin{remark}
It is easy to give an example where this bound is achieved. Let $\cK$
be the set of finite subgroups $H^s(1, n)$, and $\cK^{\hash}$ be the
h-closure (so that $\cK^{\hash}=\cK\cup \{ H^s(1, \infty)\}$
topologised as the one-point compactification). 
We take $X$ to have the value $\Q$ on the set $\cK$  and $0$ elsewhere. 
We see $I_0X$ is constant at $\Q[c]^{\vee}$ on $\cK$, 
$\cIi \Gamma_c \prod_n\Q[c]^{\vee}$ on $H^s(1, \infty)$ and $0$ elsewhere. This means
$I_1X$ is constant at $\Sigma^2 \Q[c]^{\vee}$ on $\cK$, 
$\cIi \Gamma_c \prod_n\Q[c]^{\vee}\tensor (\Q \oplus \Sigma^2\Q)$ at
$H^s(1, \infty)$. This means the cokernel of $I_0X\lra I_2X$ is non-zero at $H^s(1, \infty)$.
\end{remark}

\bibliographystyle{plain}
\bibliography{../../jpcgbib}
\end{document}